\documentclass[letterpaper]{article}
\usepackage[utf8]{inputenc}
\usepackage{amsfonts,amssymb,amsbsy,amsmath} 
\usepackage{bm} 
\usepackage[inline,shortlabels]{enumitem} 
\usepackage{listings} 
\usepackage{amsthm} 
\usepackage{commath} 
\usepackage{booktabs} 
\usepackage[top=1in, bottom=1.25in, left=1.25in, right=1.25in]{geometry}
\usepackage{relsize}
\usepackage{tikz-cd}
\usepackage{tikz}
\usepackage{booktabs}
\usepackage{verbatim}
\usepackage{algorithm,algpseudocode}
\usepackage{todonotes}
\usepackage{stmaryrd}
\usepackage{mathtools}
\usepackage[hidelinks]{hyperref}
\usepackage{cleveref} 
\usepackage[noadjust]{cite}
\usetikzlibrary{patterns}
\usepackage{combelow}

\theoremstyle{plain}
\newtheorem{theorem}{Theorem}[section]
\newtheorem{lemma}[theorem]{Lemma}
\newtheorem{corollary}[theorem]{Corollary}
\newtheorem{proposition}[theorem]{Proposition}

\theoremstyle{definition}
\newtheorem{definition}[theorem]{Definition}
\newtheorem{example}[theorem]{Example}
\newtheorem{question}[theorem]{Question}

\theoremstyle{remark}
\newtheorem{remark}[theorem]{Remark}
\newtheorem{assumption}[theorem]{Assumption}

\newcommand{\RR}{\mathbb{R}}
\newcommand{\CC}{\mathbb{C}}

\newcommand{\FF}{\mathbb{F}}

\DeclareMathOperator{\im}{im}

\DeclareMathOperator{\Hom}{Hom}
\DeclareMathOperator*{\Ass}{Ass}
\DeclareMathOperator{\Ann}{Ann}

\DeclareMathOperator{\Fitt}{Fitt}

\DeclareMathOperator{\depth}{depth}
\newcommand{\loc}[1]{\left(#1\right)_{(0)}}

\newcommand{\R}{\mathbb{R}}

\newcommand{\C}{\mathbb{C}}

\newcommand{\rank}{\operatorname{rank}}
\newcommand{\wstar}{\stackrel{*}{\rightharpoonup}}

\renewcommand{\geq}{\geqslant}

\renewcommand{\leq}{\leqslant}
\newcommand{\imag}{\operatorname{i}}

\newcommand{\per}{\operatorname{per}}

\usepackage{tikz-cd}
\title{Syzygies, constant rank, and beyond}
\author{Marc Härkönen\thanks{Georgia Institute of Technology, Atlanta, GA, USA, \url{harkonen@gatech.edu}}, Lisa Nicklasson\thanks{Università degli Studi di Genova, Genova, Italy, \url{nicklasson@dima.unige.it}}, Bogdan Rai\cb{t}ă\thanks{Scuola Normale Superiore, Pisa, Italy, \url{bogdan.raita@sns.it}}}

\begin{document}
	\maketitle
	
	\begin{abstract}
		\noindent We study linear PDE with constant coefficients. The constant rank condition on a system of linear PDEs with constant
		coefficients is often used in the theory of compensated compactness.
		While this is a purely linear algebraic condition, the nonlinear algebra
		concept of primary decomposition is another important tool for
		studying such system of PDEs. In this paper we investigate the
		connection between these two concepts.
		From the nonlinear analysis point of view, we make some progress in the study of weak lower semicontinuity of integral functionals defined on sequences of PDE constrained fields, when the PDEs do not have constant rank.
	\end{abstract}
	\section{Introduction} 
	We start from the following vague question:
	\begin{quote}
		To what extent can we \emph{solve} a system of linear PDEs with constant coefficients, $Av=0$?
	\end{quote}
	Of course, the answer depends greatly on the solution space for the unknown $v$. For instance, if we solve over the space of test functions $C_c^\infty(\R^n)$, it has been known for a long time \cite{M60} that the system admits a vector potential $v=Su$, where $S$ is another linear differential operator with constant coefficients, namely the syzygy matrix of $A$. Moreover, the operator $S$ can be computed using standard Gr\"obner basis techniques, implemented in computer algebra software such as Macaulay2 \cite{M2}. If we are interested in smooth solutions,  i.e.  in $C^\infty(\R^n)$, the existence of vector potentials is no longer guaranteed, though the extent to which the image of $S$ in this space describes the solution set is encoded in the algebraic properties of the associated polynomial module. These algebraic properties are reflected in decompositions of the solution space, see \cite{shankar_notes} and our presentation in \Cref{sec:C-UC}. 
	
	Another angle from which we would  like to investigate linear systems of PDE comes from the analysis of continuum mechanics problems, where one often studies a nonlinear relation without derivatives, coupled with a linear PDE \cite{Tartar2005}; this is the so called theory of compensated compactness \cite{Tartar1979,Murat1981,FonsecaMuller1999,GuerraRaita2019}. Since the natural  mode of convergence for such problems is a convergence of measurements (averages), spaces of smooth functions are not spaces where one can expect existence of solutions for the nonlinear problems. One rather works with spaces defined by integration, such as the Lebesgue spaces $L^p(\R^n)$. In this case, one would like to have solutions of $Av=0$ that satisfy suitable integral estimates; in this case it is reasonable to work with  (row) homogeneous operators $A$, an assumption that we keep in the introduction. For the framework of compensated compactness, the class of linear PDEs that was studied most is that of (real) constant rank operators, as it gives rise to good integral estimates.
	We say that an operator has $\mathbb F$-constant rank, $\mathbb F\in\{\R,\C\}$, if $\rank A(x)$ is constant for all $0\neq x\in\mathbb F^n$, also see Definition~\ref{def:CR}.
	In a sense, the class of real constant rank operators is the largest class where we can hope for standard harmonic analysis estimates \cite{GuerraRaita2020}. However, the existing study of this pointwise condition on the evaluations $A(x)$ falls well under the limitations of linear algebra \cite{FonsecaMuller1999,Raita2019,Raita2021}. 
	
	Thus, one of the main questions we address here is how can we link the nonlinear algebra concepts of primary decomposition to a condition that was, so far, viewed only through a linear lens.  One somewhat surprising fact is that there exists real or complex constant rank operators that do not admit a vector potential in $C^\infty(\R^n)$; such operators are as simple as the Laplacian $x_1^2+x_2^2$ or the gradient operator. On the other hand, any real constant rank operator admits a vector potential in the space
	$
	C^\infty_{\#}(Q)=\{f\in C^\infty(\R^n)\colon f\text{ is $Q$-periodic},\,\int_{Q}f=0\}
	$ of periodic functions of zero average on the cube $Q=(0,1)^n$ \cite{Raita2019}. Interestingly, the vector potential constructed there is not necessarily given by $S$, but can be replaced by it, see \Cref{cor:const_rank_exact_all_nonzero} and the discussion thereafter.
	
	Another idea used in analysis \cite{Murat1981,FonsecaMuller1999,Raita2019} formally gives a decomposition that looks very similar to the controllable-uncontrollable decomposition \eqref{eq:CUC} below: If $A$ has real constant rank and we are looking for $L^p(\Omega)$ solutions to $Av=0$ in a bounded domain $\Omega$, we can write
	\begin{align}\label{eq:error}
		v=Su+\texttt{error}
	\end{align}
	where $u\in C_c^\infty(\Omega)$ and the \texttt{error} is negligible in $L^p$. This does \emph{not} match the algebraic decomposition
	\begin{align}\label{eq:CUC}
		\ker_{C^\infty}A=\im_{C^\infty} S+\ker_{C^\infty} A_u,
	\end{align}
	where $A_u$ is an operator with trivial syzygy matrix. Although the similarity  between \eqref{eq:error} and \eqref{eq:CUC} is striking, not much has been done to explore possible connections. In Theorem~\ref{thm:intro_real_elliptic}, we make a first step in this direction, by proving that the constant rank condition of $A$ implies the  ellipticity of $A_u$. 
	%
	\begin{theorem}\label{thm:intro_real_elliptic}
		Let $A$ have real (resp. complex) constant rank. Then $A$ has a controllable-uncontrollable decomposition as in \eqref{eq:CUC} with real (resp. complex) elliptic $A_u$.
	\end{theorem}
	The converse is not true, as can be seen from Examples \ref{exa:Euler} and \ref{exa:ctrl_not_RCR}. Details on the decomposition \eqref{eq:CUC} can be found in \Cref{sec:C-UC}; the relevant definitions are in Subsection~\ref{ssec:operators}. 
	
	This result bridges the gap between \eqref{eq:error} and \eqref{eq:CUC} in the following sense: if the relation in \eqref{eq:CUC} would extend, say by approximation, to locally integrable vector fields $v$, we would have that 
	$$
	v=Su+f,\quad\text{where }A_u f=0.
	$$
	By ellipticity of $A_u$, $f$ is real analytic, so all the roughness of $v$ is carried by the potential part, $Su$. This is also the phenomenon we encounter in \eqref{eq:error}. Moreover, if \eqref{eq:CUC} extends to periodic fields, we would immediately retrieve the fact that $S$ is a vector potential for $A$ in $C^\infty_\#$, retrieving \cite[Lem.~2]{Raita2019}. This is so because the null space of an elliptic operator, as is $A_u$, is trivial  over $C^\infty_\#$.
	
	Our approach to prove \Cref{thm:intro_real_elliptic} consists of linking the controllable-uncontrollable decomposition \eqref{eq:CUC} to pointwise properties of the evaluation map of $A$. The following is the main novelty of our work:
	\begin{theorem}\label{thm:main}
		Let $A$ be a polynomial matrix with complex coefficients that has a controllable-uncontrollable decomposition as in \eqref{eq:CUC}. Then 
		$$\ker_\C A_u(\xi)=\{0\}\quad \text{for all }\xi\in \C^n\text{ such that }\rank_\C A(\xi)=\rank_{\C[x]} A.$$
	\end{theorem}
	We prove this result using tools from commutative algebra in \Cref{sec:C-CR}, see \Cref{thm:CR-elliptic}.
	Using the same tools, we can also derive the complex version that improves the real result proved recently in \cite{Raita2021} using linear algebra techniques:
	\begin{theorem}\label{thm:main_eval}
		Let $A$ be a polynomial matrix with complex coefficients and syzygy matrix $S$. Then
		$$
		\ker_\C A(\xi)={\im}_\C S(\xi)\quad\text{for all }\xi\in \C^n\text{ such that }\rank_\C A(\xi)=\rank_{\C[x]} A.
		$$
		If $\rank_\C A(\xi)<\rank_{\C[x]} A$, we have $\ker_\C A(\xi)\supsetneq \im_\C S(\xi)$.
	\end{theorem}
	It is our hope that by better understanding the controllable-uncontrollable decomposition \eqref{eq:CUC} pointwise, we will be able to derive \emph{nonstandard} quantitative estimates for solutions of general  systems of linear PDE.
	
	As a simple consequence, we note that if $A$ has complex constant rank, then the pointwise exact relation $\ker_\C A(\xi)={\im}_\C S(\xi)$ holds for all nonzero $\xi\in\C^n$.
	More generally \Cref{thm:main_eval} tells us precisely for which points the exact relation
	\begin{align*}
		\ker_{C_c^\infty}A=\im_{C_c^\infty} S
	\end{align*}
	translates to an exact relation for the evaluations of $A$ and $S$.  \Cref{thm:main_eval} also implies the main results in \cite{Raita2019,Raita2021}.
	
	Another notable consequence of Theorem~\ref{thm:main_eval} is the characterization of the  \emph{uncontrollable} operators, $A_u$ from \eqref{eq:CUC}, which are characterized by the fact that they have trivial syzygy matrix, $\ker_{\C[x]}A_u=\{0\}$. This is in turn equivalent with $\ker_\C A_u(\xi)=\{0\}$ for $\xi$ outside a proper real/complex variety a proper real/complex variety, see also \Cref{cor:uncontrollable}.
	
	The controllable-uncontrollable decomposition also provides a finer lens for the study of \emph{wave solutions}.
	In \cite{harkonen21} authors describe projective varieties that characterize possible wave solutions.
	These are constructed by looking at how $\ker_\CC A(\xi)$ changes when $\xi \in \pi$, and $\pi$ ranges through all real linear spaces of a certain specified dimension.
	In particular, \Cref{thm:main} shows that the uncontrollable part may contribute wave solutions only if the proper algebraic variety where the rank of $A(\xi)$ drops contains real linear spaces.
	A future research question would be to characterize the waves that arise from controllable and uncontrollable differential operators.
	
	Lastly, we will use some of our understanding of the algebraic properties of linear PDE to derive new results in the theory of compensated compactness. The question we are interested in is that of lower semicontinuity of integral functionals with respect to weakly convergent sequences $v_j$ that satisfy $Av_j=0$. This problem was solved completely under the real constant rank assumption in \cite{Dacorogna1982,FonsecaMuller1999,GuerraRaita2019}. However, in the absence of the rank condition, very little is known \cite{Muller1999,LeeMullerMuller2009}. We present the question in detail in Section~\ref{sec:cov} and make some progress in this unexplored direction in Theorem~\ref{thm:lsc_main}.

	The paper is organized as follows.
	In \Cref{sec:prelim} we introduce our setup and recall definitions and useful results, along with several concrete examples.
	In Subsection~\ref{ssec:modules} we describe how PDE correspond to polynomial modules, and connect algebraic properties of the module with analytic properties of the solution set.
	The concepts of ellipticity and constant rank are introduced in Subsection~\ref{ssec:operators}.
	Primary decomposition of modules is used in \Cref{sec:C-UC} to describe the controllable-uncontrollable decomposition from control theory.
	The decomposition splits the solutions obtained from the syzygy matrix, the controllable part, from the rest of the solutions, the uncontrollable part.
	We begin \Cref{sec:C-CR} by studying pointwise evaluations of the operators obtained from the controllable-uncontrollable decomposition.
	Our focus will be points where the rank of $A$ drops, which will lead to \Cref{thm:main}.
	\Cref{thm:intro_real_elliptic} then follows from the fact that constant rank operators do not drop rank outside the origin.
	We also prove Theorem~\ref{thm:main_eval}, recasted as Theorem~\ref{thm:pointwise_exact}, in Section~\ref{sec:C-CR}. In \Cref{sec:cov} we shift gears to discuss some problems from the calculus of variations under linear PDE constraints that are \emph{not} of constant rank.

	\subsection*{Acknowledgement} The authors would like to thank the Max Planck Institute for Mathematics in the Sciences, where most of this research was conducted, for the support and resources. They also thank B. Sturmfels and D. Agostini for numerous helpful suggestions. 
	M.H. is partially supported by NSF DMS-1719968 and NSF DMS-2001267.
	
	\newpage
	\section{Preliminaries}\label{sec:prelim}
	We will work over the polynomial ring $R=\mathbb F[x_1,\dotsc,x_n]$, where $\mathbb{F}\in\{\R,\C\}$. For the algebraic sections we will mostly consider $\mathbb F=\C$, whereas in the analytic Section~\ref{sec:cov} we will only investigate PDEs with real coefficients, $\mathbb F=\R$. Our PDE operators will be identified with polynomial matrices $A\in R^{\ell\times k}$, where the symbol $x_i$ corresponds to the operator $\partial_i = \frac{\partial}{\partial z_i}$. For most of the paper, the differential operators act on complex valued functions or distributions defined on $\R^n$, $v\in C^\infty(\R^n,\C)^k$ or $v\in \mathcal D'(\R^n,\C)^k$. Other spaces of functions or distributions will be considered as well. In particular, in Section~\ref{sec:cov}, the case of complex valued vector fields reduces to the study of real valued vector fields since the differential operators there have real coefficients.
	To solve the PDE $Av = 0$ means finding a $k$-tuple of $n$-variate functions $v = v(z_1,\dotsc,z_n)$ (or distributions on $\RR^n$), chosen from a suitable space of functions, such that the $\ell$ equations
	\begin{align*}
		\sum_{i=1}^k A_{ji}(\partial_1,\dotsc,\partial_n)v_i = 0
	\end{align*}
	are satisfied for all $j=1,\dotsc,\ell$.
	If $\mathcal{F}$ is a space of functions, the set of $v \in \mathcal{F}^k$ such that $Av = 0$ is denoted $\ker_\mathcal{F} A$.
	
	In some results, we will make the following homogeneity assumption:
	\begin{assumption}[Homogeneity]\label{ass:hom}
		We say that $A$ is (row-)homogeneous if for each $i=1,\ldots, k$ there exist integers $d_i$ such that $A_{ij}$ is homogeneous of degree $d_i$ for each $j=1,\ldots,\ell$.
	\end{assumption}
	
	\subsection{Modules over Polynomial Rings}\label{ssec:modules}
	The polynomial matrix $A$ describes a morphism of $R$-modules $A \colon R^k \to R^\ell$.
	Its kernel $\ker_R A$ is a finitely generated $R$-module, generated by $\{s_1,\dotsc,s_{k'}\} \subseteq R^k$.
	This in turn describes a morphism $S \colon R^{k'} \to R^k$, where $S$ is the matrix whose columns are $s_1, \dotsc, s_{k'}$.
	The matrix $S$ is called the \emph{syzygy matrix} of $A$, and makes the sequence
	\begin{align}\label{eq:syzygy}
		R^{k'} \xrightarrow{S} R^k \xrightarrow{A} R^\ell
	\end{align}
	exact, i.e. $\ker_R A = \im_R S$.
	
	Let $\mathcal{F}$ be an $R$-module, for example one of the spaces of functions or distributions discussed above.
	Tensoring the sequence \eqref{eq:syzygy} by $\mathcal{F}$ yields the complex
	\begin{align}\label{eq:syzygy_tensored}
		\mathcal{F}^{k'} \xrightarrow{S} \mathcal{F}^k \xrightarrow{A} \mathcal{F}^\ell,
	\end{align}
	which implies that $\im_\mathcal{F} S \subseteq \ker_\mathcal{F} A$.
	Note that the inclusion may be strict, as the sequence \eqref{eq:syzygy_tensored} need not be exact.
	In 1960, Malgrange \cite{M60} showed that the space of test functions $C_c^\infty$ is a \emph{flat} $R$-module, that is tensoring by $C_c^\infty$ preserves exactness.
	As a result of the exactness of \eqref{eq:syzygy_tensored} when $\mathcal{F} = C_c^\infty$, solving $Av = 0$ over $C_c^\infty$ boils down to a simple syzygy computation, a standard operation in commutative algebra. Due to its significance, we record this result here:
	\begin{theorem}\label{thm:exact_Ccinfty}
		Let $A\in R^{\ell \times k}$ be a polynomial matrix and $S$ be its syzygy matrix. Then
		$$
		{\ker_{C_c^\infty(\R^n)}}A=\im_{C_c^\infty(\R^n)}S.
		$$
	\end{theorem}
	
	In contrast, the function spaces $C^\infty$ and $\mathcal{D}'$ are certainly not flat, as for example the PDE $\partial_1 v = 0$ has a solution $v = \exp(z_2)$, but the map $x_1 \colon R \to R$ is injective.
	Instead, one of the consequences of the Ehrenpreis-Palamodov fundamental principle \cite{ehrenpreis_book, palamodov_book} is that $C^\infty$ and $\mathcal{D}'$ are \emph{injective cogenerators}, a property that induces a strong duality between $R$-submodules of $R^k$ and sets of solutions $\ker_\mathcal{F} A$ (sometimes called \emph{systems} in the control theory literature).
	We will mention properties of injective cogenerators as needed. For a more detailed account, see the article by Oberst \cite{oberst90}.
	
	Suppose $\mathcal{F}$ is an injective cogenerator.
	The transpose of $A$ and $S$ give the complex
	\begin{align}\label{eq:transposed_syzygies}
		R^{k'} \xleftarrow{S^\top} R^k \xleftarrow{A^\top} R^\ell,
	\end{align}
	which again may not be exact.
	Applying $\Hom_R(\;\cdot\;, \mathcal{F})$ to \eqref{eq:transposed_syzygies} 
	gives the complex \eqref{eq:syzygy_tensored}
	which is exact if and only if \eqref{eq:transposed_syzygies} is.
	The extent to which the complexes \eqref{eq:syzygy_tensored} and \eqref{eq:transposed_syzygies} fail to be exact can be characterized purely algebraically.
	We start with a few necessary definitions.
	\begin{definition}
		Let $U$ be an $R$-module.
		An element $u \in U$ is \emph{torsion} if there is some nonzero $r \in R$ such that $ru = 0$.
		The module $U$ is said to be torsion if all of its elements are torsion.
		The module $U$ is \emph{torsion-free} if none of its nonzero elements are torsion.
	\end{definition}
	\begin{definition}
		Let $M$ be an $R$-submodule of $R^k$.
		The prime ideal $P\subseteq R$ is an \emph{associated prime} of $R^k/M$ if there is some $u \in R^k$ such that $u \notin M$ and $P = (M \colon u) := \{r \in R \colon ru \in M \}$. We denote by $\Ass(R^k/M)$ the set of associated primes of $R^k/M$. We say that $R^k/M$ is $P$-primary if $\Ass(R^k/M) = \{P\}$.
	\end{definition}
	
	We use the notation $(f_1, \ldots, f_r)$ for the ideal of $R$ generated by the polynomials $f_1, \ldots, f_r$. In particular $(0)$ denotes the ideal consisting only of the zero polynomial.
	
	The connection between associated primes and the concepts of torsion and torsion-free modules is described the following two lemmas. For our purposes we state the lemmas in terms of the module $R^k/\im A^\top$.
	This will allow us to connect torsion and torsion-free modules to controllable and uncontrollable systems, see \Cref{sec:C-UC}.
	
	\begin{lemma}\label{lem:controllable}
		Let $A\in R^{\ell \times k}$ be a polynomial matrix and $S$ its syzygy matrix. Let $\mathcal F$ be an injective cogenerator, for instance $C^\infty(\R^n)$ or $\mathcal D'(\R^n)$. The following are equivalent:
		\begin{enumerate}
			\item The sequence \eqref{eq:syzygy_tensored} is exact, i.\,e.\ ${\ker_{\mathcal F}}A=\im_{\mathcal F}S$.
			\item The module $R^k/\im A^\top$ is torsion-free.
			\item The module $R^k/\im A^\top$ is $(0)$-primary, or trivial.
		\end{enumerate}
	\end{lemma}
	We say that a system $\ker_{\mathcal F}A$ satisfying the assumptions of Lemma \ref{lem:controllable} is \emph{controllable}.
	\begin{proof}
		A proof of the equivalence between 1.\ and 2.\ can be found in  \cite[Prop.~2.1]{shankar_notes}.
		For the equivalence of 2.\ and 3.\ let $P$ be an associated prime of $R^k/\im A^\top$. By definition $P=(\im A^\top:u)$ for some $u \in R^k$. We may also express this as 
		$$P=\{r \in R \ : \ ru=0 \ \mbox{in} \ R^k/\im A^\top \} .$$
		Hence $P$ contains nonzero elements if and only if $u$ is torsion, when considered as an element of $R^k/\im A^\top$. We can conclude that the module $R^k/\im A^\top$ is torsion-free if and only in $(0)$ is the only associated prime. 
	\end{proof}
	\begin{lemma}\label{lem:torsion}
		Let $A\in R^{\ell \times k}$ be a polynomial matrix and $S$ its syzygy matrix. The following are equivalent:
		\begin{enumerate}
			\item The module $R^k/\im A^\top$ is torsion.
			\item The ideal $(0)$ is not an associated prime of the module $R^k/\im A^\top$.
		\end{enumerate}
	\end{lemma}
	We say that a system $\ker_{\mathcal F}A$ satisfying the assumptions of Lemma \ref{lem:torsion} is \emph{uncontrollable}. We will show in Section \ref{sec:C-UC} that uncontrollability of $A$ is also  equivalent with $\ker_R A={0}$ (or $S=0)$.
	\begin{proof}
		As we noted in the proof of Lemma \ref{lem:controllable}, an associated prime $(\im A^\top : u)$ contains nonzero elements if and only if $u$ is torsion. Hence $R^k/\im A^\top$ is torsion if and only if all of its associated primes contains nonzero elements.
	\end{proof}
	
	Let $M$ be an $R$-submodule of $R^k$, and let $\Ass(R^k/M) = \{P_1,\dotsc,P_s\}$.
	Since $R$ is Noetherian, $M$ admits a (minimal, irredundant) \emph{primary decomposition}
	\begin{align}\label{eq:prim_dec}
		M = \bigcap_{i=1}^s M_{i},
	\end{align}
	where each $R^k/M_{i}$ is $P_i$-primary.
	Note that each $M_i \subseteq R^k$ is finitely generated, so $M_i = \im_R A_i^\top$ for some matrix $A_i$.
	Since $\mathcal{F}$ is an injective cogenerator, the primary decomposition translates to a decomposition of the solution space \cite{oberst90, shankar_notes}. If $M = \im_R A^\top$, we have
	\begin{align*}
		\ker_\mathcal{F} A = \sum_{i=1}^s \ker_\mathcal{F} A_{i}.
	\end{align*}
	
	\begin{remark}
		Primary decomposition of modules is built into Macaulay2 since version 1.17.
		If \texttt{M} is a submodule of $\mathtt{R}^\mathtt{k}$, e.g. obtained from \texttt{M = image transpose A}, where \texttt{A} is the matrix corresponding to the PDE $Av=0$, a list of matrices $\{A_i\}_{i=1}^s$ can be obtained by running the commands
		\begin{verbatim}
			primaryDecomposition comodule M /
			(N -> image generators N + image relations N) /
			mingens /
			transpose
		\end{verbatim}
	\end{remark}
	
	\begin{example}\label{ex:primary_decomp}
		Let $R=\CC[x,y,z]$, and consider the PDE given by
		\begin{align*}
			A = \begin{pmatrix}
				0 & -xz^2 & xy^2 \\
				-x^2y^2 & x^4 & 0  \\
				-xyz^2 & z^2 & x^3y-y^2 \\
				-x^2z^2 & 0 & x^4
			\end{pmatrix}
		\end{align*}
		The module $R^3/\im_R A^\top$ has two associated primes, namely $(0)$ and $(x)$.
		The solution set decomposes into $\ker_{C^\infty} A = \ker_{C^\infty} A_1 + \ker_{C^\infty} A_2$, where
		\begin{align*}
			A_1 = \begin{pmatrix}
				0 & -z^2 & y^2 \\
				-z^2 & 0 & x^2 \\
				-y^2 & x^2 & 0
			\end{pmatrix} && A_2 = \begin{pmatrix}
				x^2 & 0 & 0 \\
				0 & x^2 & 0 \\
				0 & 0 & x^2 \\
				0 & -xz^2 & xy^2 \\
				xyz^2 & -z^2 & y^2
			\end{pmatrix}
		\end{align*}
		Using techniques from \cite{harkonen21, manssour21}, we note that $\ker_{C^\infty} A_1 = \im_{C^\infty} S$, where
		\begin{align*}
			S = \begin{pmatrix}
				x^2 \\ y^2 \\ z^2
			\end{pmatrix},
		\end{align*}
		and $\ker_{C^\infty} A_2$ consists of functions of the form
		\begin{align*}
			\begin{pmatrix}
				\phi(b,c) \\ 0 \\ 0
			\end{pmatrix} \qquad \text{ and } \qquad \begin{pmatrix}
				a\psi(b,c) \\ \frac{\partial \psi}{\partial b}(b,c) \\ 0
			\end{pmatrix},
		\end{align*}
		where $x,y,z$ act as $\frac{\partial}{\partial a},\frac{\partial}{\partial b},\frac{\partial}{\partial c}$ respectively, and $\phi$, $\psi$ are smooth functions $\RR^2 \to \CC$.
		Note further that $C^\infty$ can be replaced by $\mathcal{D}'$, or indeed any injective cogenerator.
	\end{example}

	\subsection{Classes of Operators}\label{ssec:operators}
	Our general goal  is to convert algebraic properties of the polynomial matrix $A$ into analytic properties of the system of PDEs $A v=0$. To this end, we will only focus on homogeneous systems, i.e. for the remainder of this subsection, operators $A$ are assumed to satisfy Assumption~\ref{ass:hom}.
	
	The following ellipticity conditions are well understood analytically:
	\begin{definition}
		Let $\mathbb{F}\in\{\R,\C\}$. We say that $A$ is \emph{$\mathbb F$-elliptic} if $\ker_\C A(\xi)=\{0\}$ for all $\xi\in\mathbb F^n\setminus\{0\}$.
	\end{definition}
	It is a classical result, see e.g., \cite{Hormander}, that $\R$-ellipticity of $A$  is equivalent to analyticity of all distributional solutions of $Av=0$. $\C$-ellipticity is also well understood \cite{Smith1970,GRVS,oberst90,manssour21} and is equivalent to the fact that all solutions of $A v=0$ are not only analytic, but actually polynomials. We will revisit aspects of these results later, in Section \ref{sec:C-CR}.
	
	Another important property is that of constant rank, which is particularly relevant in the study of compensated compactness \cite{Murat1981,FonsecaMuller1999,GuerraRaita2019}. 
	\begin{definition}\label{def:CR}
		An operator $A$ is said to be of \emph{$\mathbb F$-constant rank} if there exists an integer $r$ such that $\mathrm{rank}_\C\, A(\xi)=r$ for all $\xi\in\mathbb F^n\setminus\{0\}$.
	\end{definition}
	
	For $A$ to be $\FF$-elliptic it is necessary that $k \le \ell$. If $A$ is $\FF$-elliptic then $A$ has $\FF$-constant rank $k$. 
	
	The class of $\R$-constant rank operators is, roughly speaking, the largest class where standard harmonic analysis results hold, see \cite{FonsecaMuller1999,GuerraRaita2020}. To the best of our knowledge, the $\C$-constant rank condition has not been used in the analysis literature, but it is algebraically easier to handle than the $\RR$-constant rank condition. We will explore this seemingly new condition in detail in Section~\ref{sec:C-CR}. 
	
	We will conclude this subsection with a few examples that illustrate the differences between these conditions. Part of the aim of this paper will be to compare these pointwise  conditions on the evaluations, that come from the ``analysis with estimates'' of the linear PDE systems, with natural conditions that come from the algebraic geometry angle; for instance, see Section~\ref{sec:C-UC} for the notions of \emph{controllability} and \emph{uncontrollability}; cf. Lemmas \ref{lem:controllable} and \ref{lem:torsion}.
	\begin{example}
		The operators
		$$
		A_1=\left( 
		\begin{matrix}
			x_1&0\\
			0&x_2
		\end{matrix}
		\right),\quad A_2=x_1^2-x_2^2
		$$
		do not have $\R$-constant rank. Additional examples can be found in Examples \ref{exa:Euler} and \ref{exa:ctrl_not_RCR}.
	\end{example}
	\begin{example}
		The operators
		$$
		A_3=x_1^2+x_2^2,\quad A_4=
		\left(
		\begin{matrix}
			x_1^2+x_2^2&0&-x_1^2-x_3^2\\
			0&-x_2^2-x_3^2&x_1^2+x_3^2\\
			-x_1^2-x_2^2&x_2^2+x_3^2&0
		\end{matrix}
		\right)
		$$
		have $\R$-constant rank but fail to have $\C$-constant rank. In fact, $A_3$ is $\R$-elliptic but not $\C$-elliptic. 
	\end{example}
	\begin{example}
		The operator
		$$
		A_5=\left( 
		\begin{matrix}
			0&x_3&-x_2\\
			-x_3&0&x_1\\
			x_2&-x_1&0
		\end{matrix}
		\right)
		$$
		is of $\C$-constant rank but not $\R$-elliptic.
	\end{example}
	\begin{example}
		The operators
		$$
		A_6=\left( 
		\begin{matrix}
			x_1\\
			x_2
		\end{matrix}
		\right),\quad  A_7=\left( 
		\begin{matrix}
			x_1^2+x_2^2\\
			x_1^2-x_2^2
		\end{matrix}
		\right), \quad  A_8=\left( 
		\begin{matrix}
			x_1&0\\
			0&x_2\\
			x_2&x_1
		\end{matrix}
		\right)
		$$
		are $\C$-elliptic.
	\end{example}
	\section{Controllable--Uncontrollable Decomposition}\label{sec:C-UC}
	As we saw in \Cref{lem:controllable} the syzygy matrix $S$ describes all smooth solutions to the PDE $Av = 0$ if and only if the corresponding quotient module $R^k/\im_R A^\top$ is torsion-free, since the space of smooth functions $C^\infty$ is an injective cogenerator.
	We recall that in control theory such a system $\ker_\mathcal{F} A$ is said to be controllable, while at the opposite end of the spectrum, the system $\ker_\mathcal{F} A$ is said to be uncontrollable when the quotient module is torsion \cite{shankar_notes}.
	We remark that ``uncontrollable'' is not the same as ``not controllable'', since an $R$-module $M$ can have a set of torsion elements that is a nontrivial, strict subset of $M$.
	By exploiting the primary decomposition, we can decompose any solution space into two subspaces, one of which is controllable and the other one uncontrollable.
	\begin{proposition}[Controllable-uncontrollable decomposition]\label{prop:decomp}
		Let $\mathcal{F} = C^\infty(\R^n)$ or $\mathcal{D}'(\RR^n)$ (or any injective cogenerator), and $A \in R^{\ell \times k}$.
		There exist polynomial matrices $A_c, A_u, S$ such that
		we have a decomposition
		\begin{align*}
			\ker_\mathcal{F} A = {\im}_\mathcal{F} S + \ker_\mathcal{F} A_u,
		\end{align*}
		where
		\begin{enumerate}
			\item $\ker_\mathcal{F} A_c = \im_\mathcal{F} S$,
			\item $R^k/\im A_c^\top$ is either $(0)$-primary or trivial,
			\item the prime $(0)$ is not an associated prime of $R^k/\im A_u^\top$,
			\item $\im_R A^\top = \im_R A_c^\top \cap \im_R A_u^\top$ as $R$-modules,
			\item $\ker_R A_u = 0$, i.e. $\im_R A_u$ is a free $R$-module,
			\item $\Ass(R^k/\im A^\top) = \Ass(R^k/\im A_c^\top) \cup \Ass(R^k/\im A_u^\top)$, and the union is disjoint.
		\end{enumerate}
		In particular, the system $\ker_\mathcal{F} A_c$ is controllable, and $\ker_\mathcal{F} A_u$ is uncontrollable.
	\end{proposition}
	\begin{proof}
		Our point of departure is the primary decomposition \eqref{eq:prim_dec}.
		We write
		\begin{align*}
			\im A^\top = M_c \cap M_u,
		\end{align*}
		where $M_c \subseteq R^k$ is the $(0)$-primary component, and $M_u \subseteq R^k$ is the intersection of all other primary components.
		If there are no $(0)$-primary components (or if $\im A^\top$ is $(0)$-primary), then $M_c$ (or $M_u$) is equal to $R^k$.
		Let $A^\top_c, A^\top_u$ be polynomial matrices such that $M_c = \im A^\top_c$ and $M_u = \im A_u^\top$.
		Since $\mathcal{F}$ is an injective cogenerator, we have the decomposition
		\begin{align*}
			\ker_\mathcal{F} A = \ker_\mathcal{F} A_c + \ker_\mathcal{F} A_u.
		\end{align*}
		If we choose $S$ to be the syzygy matrix of $A_c$, we obtain the required decomposition.
		Properties 1, 2, 3, 4, 6 follow by construction.
		
		It follows from 2.\ and \Cref{lem:controllable} that  $\ker_\mathcal{F} A_c$ is controllable, and from 3.\ and \Cref{lem:torsion} that $\ker_\mathcal{F} A_u$ is uncontrollable.
		
		For property 5.\ suppose that $\ker_R A_u = \im_R S_u$ for some nonzero matrix $S_u$.
		Then for any nonzero compactly supported $w$ the function $S_u w$ is a nonzero compactly supported solution in $\ker_\mathcal{F} A_u$.
		This is a contradiction, as it follows from the Paley-Wiener Theorem that uncontrollable systems don't contain compactly supported solutions, c.\,f.\ \cite[Prop~3.2]{shankar_notes}.
	\end{proof}
	
	We remark that in the construction above, we chose $S$ to be the syzygy matrix of $A_c$, but in fact it coincides with the syzygy matrix of $A$ itself.
	
	\begin{theorem}\label{thm:kerA=kerA_c}
		For $A$ and  $A_c$ as in Proposition \ref{prop:decomp}
		\begin{align*}
			\ker_R A = \ker_R A_c.
		\end{align*}
	\end{theorem}
	\begin{proof}
		Let $\loc{\cdot}$ denote the localization at the prime $(0)$, and recall that $R_{(0)} = \FF(x_1, \ldots, x_n)$ is the field of rational functions. As a first step we show that $\loc{\im_R A^\top_u}=R^k_{(0)}$, where $A_u$ denotes the uncontrollable part in a decomposition as in Proposition \ref{prop:decomp}. For any $v \in R^k$ the ideal $(\im_R A_u^\top : v)$ is nonzero, as $(0)$ is not an associated prime.  Hence there is a nonzero $r \in R$ such that $rv \in \im_R A_u^\top$. When we localize at $(0)$ the element $r$ becomes invertible, so $v=r^{-1}rv \in \im_R A_u^\top$. It follows that $\loc{\im_R A^\top_u}=R^k_{(0)}$. As $\im_R A^\top = \im_R A^\top_c \cap \im_R A^\top_u$ we have $\loc{\im_R A^\top} = \loc{\im_R A_c^\top}$. In the special case when $A$ in uncontrollable we get $\loc{\im_R A^\top} = R^k_{(0)}$.
		
		Since $\im_R A^\top \subseteq \im_R A_c^\top$, there is a polynomial matrix $B$ such that $A^\top = A_c^\top B$. Suppose $u \in \ker_R A_c$, then $Au = B^\top A_c u = 0$, so $u \in \ker_R A$.
		
		For the converse, since $\loc{\im_R A_c^\top} \subseteq \loc{\im_R A^\top}$, there is some matrix $C$ with entries in $R_{(0)}$ such that $A_c^\top = A^\top C$.
		Clearing denominators, we have $gA_c^\top = A^\top C'$ for some $0 \neq g \in R$ and a matrix $C'$ with entries in $R$.
		If $u \in \ker_R A$, then $g(A_c u) = {C'}^\top A^\top u = 0$, and since $g \neq 0$, we must have $A_c u = 0$.
		Hence $\ker_R A_c = \ker_R A$.
	\end{proof}

	While the controllable part is well understood as the image of the vector potential $S$, the uncontrollable part is less explored. Our Theorem~\ref{thm:main_eval} gives the following characterization of uncontrollable operators:
	\begin{corollary}\label{cor:uncontrollable}
		Let $A\in R^{\ell \times k}$ be a polynomial matrix and $S$ its syzygy matrix. Let $\mathcal F$ be an injective cogenerator, for instance $C^\infty(\R^n)$ or $\mathcal D'(\R^n)$. The following are equivalent:
		\begin{enumerate}
			\item $\ker_\mathcal{F} A$ is uncontrollable, i.e. $R^k/\im A^\top$ is torsion,
			\item $\ker_R A=\{0\}$, i.e. $S=0$,
			\item $\ker_\C A(\xi)=\{0\}$ for all $\xi\in\R^n$, except on a proper real variety,
			\item $\ker_\C A(\xi)=\{0\}$ for all $\xi\in \C^n$, except on a proper complex variety.
		\end{enumerate}
		The proper variety in each case is $\{\xi\in\mathbb{F}^n\colon \rank A(\xi)\text{ is not maximal}\}$, $\mathbb F\in\{\R,\C\}$.
	\end{corollary}
	\begin{proof}
		Suppose $\ker_R A = \im_R S \neq \{0\}$. If $u$ is any compactly supported function, then $v = Su$ is a compactly supported solution to $Av = 0$, so in particular $\ker_\mathcal{F} A$ is not uncontrollable.
		If $\ker_R A = \{0\}$, we can apply the Controllable-Uncontrollable decomposition to get $\ker_\mathcal{F} A = 0 + \ker_\mathcal{F} A_u$, so in particular $\ker_\mathcal{F} A$ is uncontrollable. This proves the equivalence between 1.\ and 2.
		
		The equivalence of statements 2., 3., and 4.\ follow from \Cref{thm:main_eval} and the fact that $\RR^n$ is not a subvariety of $\CC^n$.
	\end{proof}
	Therefore, the triviality of the syzygy matrix characterizes uncontrollability. In contrast, there is no condition on the syzygy matrix alone that can characterize controllability. This follows from Theorem \ref{thm:kerA=kerA_c}, by taking an operator $A$ that is not controllable and noticing that $\im_R S$ is then the kernel of both an operator $A_c$ that is controllable and of $A$, which is not.
	We summarize the various alternative definitions of controllability and uncontrollability in \Cref{table}.

	\begin{table}[h]
		\centering
		\begin{tabular}{|c |c |c| c| c|} 
			\hline 
			& controllable & uncontrollable\\ 
			\hline\hline
			\rule{0pt}{2.5ex} Torsion elements & $R^k/\im A^\top$ is torsion-free &  $R^k/\im A^\top$ is torsion \\[0.5ex] \hline
			\rule{0pt}{2.5ex} Associated primes & $R^k/\im A^\top$ is $(0)$-primary & $(0)\notin \mathrm{Ass}(R^k/\im A^\top$)  \\ [0.5ex]\hline
			\rule{0pt}{2.5ex}
			Vector potential & $\ker_{C^\infty(\R^n)}A= \im_{C^\infty(\R^n)}S$ & $\ker_{C_c^\infty(\R^n)}A= \{0\}$    \\[0.5ex] \hline
			\rule{0pt}{2.5ex} Syzygy matrix & \textbf{no} condition  & $S=0$   \\ [0.5ex]
			\hline
		\end{tabular}
		\caption{ \label{table} A summary of the equivalent definitions of controllable and uncontrollable operators.}
	\end{table}

	In the setting of \Cref{cor:uncontrollable}, the nature of the solutions of the PDE $Av=0$ can be very different, depending on the structure of the set of points where $\ker_{\CC}A(\xi) \ne \{0\}$.
	\begin{example}
		Consider the examples $A_2$, $A_3$, $A_7$ from Subsection~\ref{ssec:operators}. All three operators are uncontrollable with 
		$$
		\ker_R A_i = \{0\},\quad \ker_{C_c^\infty}A_i=\{0\},\quad\ker_{C^\infty}A_i\neq\{0\}.
		$$
		In each example we investigate the latter set.
		We also look at the varieties $X_\RR$, resp. $X_\CC$ where conditions 3.,\, resp. 4.\ of \Cref{cor:uncontrollable} fail.
		
		If $A=A_2=\partial_1^2-\partial_2^2$, then any function
		$v(z_1,z_2)=f(z_1\pm z_2)$ for $f\in C^\infty(\R)$ is a solution. The operator is not $\R$-elliptic.
		The varieties $X_\RR, X_\CC$ are both pairs of lines.
		
		If $A=A_3=\partial_1^2+\partial_2^2$, the solutions are of the form $v(z_1,z_2)=g(z_1\pm\imag z_2)$, where $g\in C^\infty(\C)$. The increase in regularity is substantial, particularly since, in this case, the solutions are real analytic. The operator is $\R$-elliptic, but not $\C$-elliptic.
		The variety $X_\CC$ is again a pair of lines, but now $X_\RR$ is the origin.
		
		If $A=A_7=(A_2,A_3)^\top$, we have that the solutions are $v(z_1,z_2)=az_1z_2+bz_1+cz_2+d$, which are polynomials. This is yet another increase in regularity from being analytic. The operator is $\C$-elliptic.
		Here both $X_\CC$ and $X_\RR$ are the origin.
	\end{example}
	In practice, many $\RR$-constant rank operators happen to be also $\CC$-constant rank, so the ellipticity of $A_u$ follows from the complex part of  \Cref{thm:intro_real_elliptic}.
	If $A$ is controllable, then the conclusion of \Cref{thm:intro_real_elliptic} is also trivial, as one can choose $A_u = 1$.
	We present a concrete example where the real part of \Cref{thm:intro_real_elliptic} applies nontrivially. 
	
	\begin{example}
		Let $$A = \begin{pmatrix}
			x(x^2+y^2) & y(x^2+y^2)
		\end{pmatrix}.$$
		The operator drops rank when $x^2 + y^2 = 0$, hence it has $\RR$-constant rank, but not $\CC$-constant rank.
		It is not $\RR$-elliptic either, nor is it controllable, as we have $\Ass(R^2/\im A^\top) = \{(0), (x^2+y^2)\}$.
		The controllable part is given by the operator $A_c = \begin{pmatrix}x & y\end{pmatrix}$, so that $S = \begin{pmatrix}y & -x\end{pmatrix}^\top$.
		The uncontrollable part corresponds to the operator
		\begin{align*}
			A_u = \begin{pmatrix}
				x(x^2+y^2) & -y(x^2+y^2) \\
				y(x^2+y^2) & x(x^2+y^2)
			\end{pmatrix},
		\end{align*}
		whose determinant is $(x^2+y^2)^3$, so $A_u$ is indeed $\RR$-elliptic.
	\end{example}
	By replacing the Laplacian $x^2+y^2$ with the wave operator $x^2-y^2$ in the example above, we obtain an example in which  $A_u$ is not elliptic:
	\begin{example}
		Let
		\begin{align*}
			A = \begin{pmatrix}
				x(x^2-y^2) & y(x^2-y^2)
			\end{pmatrix}
		\end{align*}
		Its rank drops whenever $x = \pm y$, so it does not have $\RR$-constant rank.
		The uncontrollable part is described by the operator
		\begin{align*}
			A_u = \begin{pmatrix}
				x(x^2-y^2) & y(x^2-y^2) \\ y(x^2 - y^2) & x(x^2 - y^2),
			\end{pmatrix}
		\end{align*}
		whose solutions take the form
		\begin{align*}
			v(a,b) = \begin{pmatrix}
				f_1(a+b) - af_3(a+b) + f_4(a-b) + af_6(a-b) \\ f_2(a+b) + af_3(a+b) + f_5(a-b) + af_6(a-b)
			\end{pmatrix},
		\end{align*}
		where $f_1,\dotsc,f_6 \in C^\infty(\RR, \CC)$.
	\end{example}

	\section{Generic Rank and Associated Primes}\label{sec:C-CR}
	As rank and ellipticity conditions require homogeneity by definition,  Assumption~\ref{ass:hom} is implicit whenever $\mathbb F$-ellipticity/constant rank is mentioned.
	Let $A_u$ denote the uncontrollable component of a decomposition of a given operator $A \in R^{\ell \times k}$, as in Proposition \ref{prop:decomp}. Moreover, we let $M$ denote the module $R^k /\im A^\top$. 
	
	The aim of this section is to prove Theorem \ref{thm:CR-elliptic} which contains the main Theorems~\ref{thm:intro_real_elliptic} and~\ref{thm:main}.

	\begin{theorem}\label{thm:CR-elliptic}
		If $A(\xi)$ has maximal rank for a point $\xi \in \CC^n$, then $\ker_\CC A_u(\xi) = \{0\}$. 
		In particular, if  
		$A$ has $\FF$-constant rank, then $A_u$ is $\FF$-elliptic, for $\FF \in \{\RR, \CC \}$.
	\end{theorem}%
	Here we clarify that the rank of an evaluation $A(\xi)$ is maximal if $\rank A(\xi)$ equals the generic rank, i.\,e.\ the maximal value of the map $\xi\mapsto \rank A(\xi)$. We emphasize that the first part of the result, concerning the the point evaluations $A(\xi)$ and $A_u(\xi)$, holds without any homogeneity restrictions on $A$. 
	
	One important observation when it comes to real constant rank is that if $A$ has $\RR$-constant rank then this rank is equal to the rank of $A$ at a generic \emph{complex} point, i.\,e.\ the maximal rank of $A$. Indeed, if the rank of $A$ were to drop for all of $\RR^n$, it would also have to drop for all points in the (complex) Zariski closure of $\RR^n$, namely all of $\CC^n$.
	
	The key to proving Theorem \ref{thm:CR-elliptic} is the following result, which we prove later in this section. We use the notation $V(I)$ for the complex algebraic variety of an ideal $I \subseteq R$, i.\,e.\ the set of common zeroes of the polynomials in $I$. 
	
	\begin{theorem}\label{thm:CR-ass}
		Let $P$ be a nonzero associated prime of $M=R^k /\im_R A^\top$. Then 
		$$ \{ \xi\in \CC^n : \ \rank A(\xi) \mbox{ is maximal} \} \cap V(P) = \emptyset.$$ 
		In particular, if $A$ has $\RR$-constant rank then the variety $V(P)$ contains no real nonzero points. If $A$ has $\CC$-constant rank, then $\Ass M \subseteq \{(0),(x_1, \ldots, x_n)\}$.
	\end{theorem}
	
	The converse implications of the two theorems are not true, as can be seen from the Euler equations as expressed in \cite{DLSZ}, presented below in two space dimensions.
	\begin{example}\label{exa:Euler}
		Let
		$$
		A=\begin{pmatrix}
			x_1 & 0 & x_2 & x_3 & x_2 \\
			0 & x_1 & -x_3 & x_2 & x_3 \\
			x_2 & x_3 & 0 & 0 & 0 
		\end{pmatrix}
		$$
		This is a controllable operator, and we can simply take $A_u$ to be multiplication by $1$ which is trivially $\C$-elliptic. The generic rank of $A$ is 3, but it drops to 2 when $x_2=x_3=0$, so $A$ does not have $\R$-constant rank. Moreover the only associated prime of the module $M$ is $(0)$. 
		
		For an example where the uncontrollable part is elliptic and nontrivial, consider the $9 \times 5$ matrix $B$ obtained from $A$ by multiplying each row by $x_1$, $x_2$, and $x_3$. Then $B$ has the same rank as $A$ in every point, and the associated primes are $(0)$ and $(x_1,x_2,x_3)$. The uncontrollable part $B_u$ is given by a $\C$-elliptic $24\times5$ matrix with entries of degree two. Its  null space over $C^\infty$ must therefore contain affine functions.
	\end{example}    
	
	Clearly any polynomial matrix $A$ has constant rank if and only if $A^\top$ does. If we consider $A^\top$ instead of $A$ in Example \ref{exa:Euler} we get the ideal $(x_2,x_3)$ as an associated prime, which describes exactly the points where the rank of $A^\top$ (and hence also $A$) drops. One might be tempted to believe that if the rank of a matrix $A$ drops from the generic rank at a point $\xi$, then $\xi$ is in the variety of an associated prime of the module defined by $A$ or $A^\top$. However this is not the case, as we see in the next example. 
	
	\begin{example}\label{exa:ctrl_not_RCR}
		The operator 
		$$A=
		\begin{pmatrix}
			x_3 &x_3 & 0 & x_3(x_3-x_4) & x_4(x_4-x_2-x_3)+x_2x_3 \\
			x_1 & 0 & x_4(x_2+x_3-x_4)-x_2x_3 &0&0 \\
			0& x_2& 0 & x_4(x_3-x_4) & 0 \\
			0 & 0 & x_2x_3 & x_1x_3 & x_1x_2
		\end{pmatrix}
		$$
		does not have constant rank, but both $R^4/\im_R A$ and $R^5/\im_R A^\top$ are $(0)$-primary. 
	\end{example}
	
	Recall \cite{manssour21} that the \emph{characteristic variety} $V(M)$ is given by 
	$$V(M) = V(\Ann_R M)= V(P_1) \cup \dots \cup V(P_s)$$
	where $\Ass(M)= \{P_1, \ldots, P_s\}$.
	Alternatively, the characteristic variety is the vanishing set of the $k \times k$ minors of $A$. Hence $\xi \in V(M)$ if and only if $A(\xi)$ has non-trivial kernel. In this way we get a characterization of $\CC$-elliptic and $\RR$-elliptic operators as stated in Proposition \ref{prop:c-elliptic_ass_primes} below. This may be compared with the description of $\CC$-elliptic operators given in  \cite[Proposition 3.2]{GRVS}.
	
	\begin{proposition}\label{prop:c-elliptic_ass_primes}
		Let $\Ass(M)= \{P_1, \ldots, P_s\}$. Then 
		$$ V(P_1) \cup \dots \cup V(P_s) = \{ \xi \in \CC^n : \ A(\xi) \ \mbox{has a nontrivial kernel}\}.$$
		In particular, the following are equivalent:
		\begin{enumerate}
			\item 
			$A$ is $\RR$-elliptic, 
			\item the varieties $V(P_i)$ contains no real nonzero points. 
			\item $\ker_{\mathcal{D}'} A$ consists only of real analytic functions. 
		\end{enumerate}
		Moreover, the following are also equivalent:
		\begin{enumerate}
			\item $A$ is $\C$-elliptic,
			\item $M$ is either trivial or $(x_1,\ldots,x_n)$-primary,
			\item $\ker_{\mathcal{D}'} A$ consists only of polynomials.
		\end{enumerate}
	\end{proposition}
	That $\R$-ellipticity of $A$  is equivalent to analyticity of all  solutions of $Av=0$ is well known \cite{Hormander}.
	\begin{proof}
		The first statement is deduced in the paragraph before the proposition. It follows directly that $A$ is $\CC$ or $\RR$-elliptic if and only if the characteristic variety contains no nonzero complex or real points respectively. In the complex case this means that the only possible associated prime is the maximal ideal $(x_1, \ldots, x_n)$, as this is the only prime ideal whose variety is the origin. 
		
		The remaining equivalence is shown using the Ehrenpreis-Palamodov fundamental principle.
		If $M$ is $(x_1,\dotsc,x_n)$-primary, all solutions to $Av=0$ are of the form
		\begin{align*}
			v(z) = \sum_i \int B_i(x,z) \exp(-\imag\langle x, z \rangle) \,d\mu_i(x),
		\end{align*}
		where the $\mu_i$ are measures supported at the origin, and $B_i(x,z)$ are polynomials.
		Therefore we must have $v(z) = \sum_i c_i B_i(0,z)$ for some constants $c_i$. This is clearly a polynomial.
		Conversely, if the characteristic variety contains a nonzero point $\xi \in \CC^n$, then there is a nonpolynomial solution, namely $v(z) = u \exp(\langle \xi, z \rangle)$ for some constant vector $u \in \CC^k$.
	\end{proof}    
	
	Recall from Proposition \ref{prop:decomp} that $\Ass(R^k/\im A_u^\top) = \Ass(M) \setminus \{(0)\}$. In combination with Proposition \ref{prop:c-elliptic_ass_primes} we obtain the following.
	
	\begin{corollary}\label{cor:Au_ellptic}
		Let $P_1, \ldots, P_r$ be the nonzero associated primes of $M$.  Then 
		$$ V(P_1) \cup \dots \cup V(P_r) = \{ \xi \in \CC^n : \ \ker_\CC A_u(\xi)\ne \{0\}\}.$$
		In particular $A_u$ is $\RR$-elliptic if and only if the varieties $V(P_i)$, $i=1, \ldots, r$, contains no real nonzero points, and $A_u$ is $\CC$-elliptic if and only if $\Ass(M) \subset \{(0), (x_1,\ldots,x_n)\}$.  
	\end{corollary}
	
	\begin{proof}[Proof of Theorem \ref{thm:CR-elliptic}]
		By Theorem \ref{thm:CR-ass} and Corollary \ref{cor:Au_ellptic} we have 
		$$ \{ \xi\in \CC^n : \ \rank A(\xi) \mbox{ is maximal} \}  \cap  \{ \xi \in \CC^n : \ \ker_\CC A_u(\xi)\ne \{0\}\} = \emptyset. \qedhere$$
	\end{proof}
	
	In preparation for the proof of Theorem \ref{thm:CR-ass}, we introduce some notation and results from \cite{eisenbud:commalg}. 
	From this point until the end of the section, no homogeneity assumption is needed. 
	In commutative algebra, \emph{Fitting ideals} are important invariants of finitely generated modules.
	\begin{definition}
		Let $\phi \colon R^{\ell} \to R^k$ be a map of free modules, described by a $k \times \ell$ matrix with entries in $R$.
		The ideal $I_j( \phi)$ is defined as the ideal generated by the $j \times j$ minors (i.\,e.\ determinants of submatrices) of the matrix representing $\phi$.
		The \emph{rank} of $\phi$, denoted $\rank_R \phi$ is the largest integer such that $I_j( \phi ) \neq (0)$.
		If $R^\ell \xrightarrow{\phi} R^k \to M \to 0$ is a presentation of the $R$-module $M$, then the $r$th \emph{Fitting ideal} is the ideal
		\begin{align*}
			\Fitt_r(M) := I_{k-r}(\phi)
		\end{align*}
		We denote by $I(M)$, or $I(\phi)$, the first nonzero Fitting ideal; note that $I(M) = I_{\rank \phi}(\phi)$.
	\end{definition}
	
	The ideal $I_j(\phi)$ of $j \times j$ minors is independent of the choice of matrix representing $\phi$. This fundamental fact is sometimes referred to as \emph{Fitting's Lemma}.  
	
	Given a matrix $A \in R^{\ell \times k}$ we can write down a presentation $R^\ell \xrightarrow{A^\top} R^k \to M \to 0$ of the module $M = R^k/\im A^\top$. Here the $r$th Fitting ideal is the ideal generated by the minors of size $k - r$ of $A$.
	Furthermore, $\rank_R A^\top$ coincides with the rank of the evaluated matrix $A(\xi)$ for a generic point $\xi \in \CC^n$.
	Thus it is easy to see that if $A$ is a $\CC$-constant rank operator, then its Fitting ideals can only be either $(0)$ or $(x_1,\dotsc,x_n)$.
	This gives us a fully algebraic characterization of PDE given by a matrix $A$ of $\CC$-constant rank: they correspond to modules $M=R^k/\im A^\top$ where $I(M) = (x_1,\dotsc,x_n)$.
	
	We begin by recalling some standard results about exact sequences in commutative algebra.
	\begin{lemma}[{\cite[Cor.~20.12]{eisenbud:commalg}}]\label{lem:exact_seq_radical}
		Suppose
		\begin{align*}
			0 \to R^{k_m} \xrightarrow{\phi_m} R^{k_{m-1}} \xrightarrow{\phi_{m-1}} \cdots \xrightarrow{\phi_2} R^{k_1} \xrightarrow{\phi_1} R^{k_0}
		\end{align*}
		is an exact sequence of free $R$-modules. Then
		\begin{align*}
			\sqrt{I(\phi_k)} \subseteq \sqrt{I(\phi_{k+1})} \quad \text{for all } k \geq 1.
		\end{align*}
	\end{lemma}
	
	Here $\sqrt{I}$ denotes the \emph{radical} of the ideal $I$, defined as the set of elements $r$ such that $r^k\in I$ for some $k$. 
	
	\begin{lemma}[{\cite[Cor.~20.14]{eisenbud:commalg}}]\label{lem:exact_seq_ass}
		Let $M$ be an $R$-module with finite free resolution
		\begin{align}\label{eq:module_free_res}
			0 \to R^{k_m} \xrightarrow{\phi_m} \cdots \to R^{k_1} \xrightarrow{\phi_1} R^{k_0} \to M \to 0.
		\end{align}
		If $P$ is a prime of $R$ and $d = \operatorname{depth}(P)$, then $P \in \Ass M$ if and only if $P \supseteq I(\phi_d)$.
	\end{lemma}
	In our setting $\depth(P)$ can be defined as the maximal length of a regular sequence inside $P$. The only prime ideal of depth zero is the ideal $(0)$, so all associated primes of $M$ are detected by \Cref{lem:exact_seq_ass}.
	
	Since a morphism of $R$-modules $\phi \colon R^\ell \to R^k$ can be represented by a $k \times \ell$ matrix of polynomials, one can also study the linear map $\phi(\xi) \colon \CC^\ell \to \CC^k$ corresponding to the evaluation of the entries of $\phi$ for any point $\xi \in \CC^n$.
	In general, while evaluations of exact sequences of $R$-modules are not necessarily exact, the set of points where the evaluated sequence fails to be exact is a proper (complex) algebraic variety.
	This means that a syzygy matrix $S$ of $A$ has the property that $\im_\CC S(\xi) = \ker_\CC A(\xi)$ for almost every $\xi \in \CC^n$. It is also clear that this holds over $\RR$.
	For a rigorous proof, see \cite[Cor.~3.4]{eisenbud:syz}.

	Now we are ready to prove Theorem \ref{thm:CR-ass}. 
	
	\begin{proof}[Proof of Theorem \ref{thm:CR-ass}]
		Let $X \subset \CC^n$ be the set of points where the rank of $A$ is maximal. Suppose $P$ is a nonzero associated prime of $M$, and let $d=\depth(P)$. Consider a minimal free resolution of $M$, i.\,e.\ a free resolution which is also exact. With the notation in (\ref{eq:module_free_res}), the map $\phi_1$ is given by the matrix $A^\top$. By Lemma \ref{lem:exact_seq_ass} we have $P \supseteq I(\phi_d)$, which implies $P\supseteq \sqrt{I(\phi_d)}$. It then follows from Lemma \ref{lem:exact_seq_radical} that $P \supseteq  \sqrt{I(\phi_1)}$, or equivalently $V(P) \subseteq V(I(\phi_1))$. But $V(I(\phi_1))$ are precisely the points $\xi$ where $\rank_\FF A(\xi) < \rank_R A$. Hence $X \cap V(P) = \emptyset$. 
	\end{proof}

	Since the converse of Theorem \ref{thm:CR-elliptic} fails, it would also be interesting to characterize when an uncontrollable operator is $\C$-elliptic more precisely.
	
	To further improve our understanding of $\C$-constant rank operators, we will give the complex version and sharpen the real result proved in \cite{Raita2021}:
	\begin{theorem}\label{thm:pointwise_exact}
		Let $A \in R^{\ell \times k}$ be a polynomial matrix with syzygy matrix $S$. Then
		$$
		\ker_\C A(\xi)={\im}_\C S(\xi)\quad\text{for all }\xi 
		\ \text{where }
		\rank A(\xi) \text{ is maximal}.
		$$
		In the set where the rank is not maximal, we have $\ker_\C A(\xi)\supsetneq \im_\C S(\xi)$.
	\end{theorem}
	
	\begin{proof}
		Take a free resolution of $R^\ell/\im_R A$ as in \cref{eq:module_free_res}.
		The maps $\phi_1, \phi_2$ orrespond to $A,S$ respectively, and the equality $k = \rank A(\xi) + \rank S(\xi)$ is true almost everywhere.
		
		In particular, the maximal ranks of $A$ and $S$ sum to $k$, so $\ker A(\xi) = \im S(\xi)$ whenever the ranks of $A$ and $S$ are maximal.
		By \Cref{lem:exact_seq_radical}, the variety where the rank of $A$ drops contains the variety where the rank of $S$ drops, hence the rank of $S$ is maximal whenever the rank of $A$ is.
	\end{proof}
	
	\begin{corollary}\label{cor:const_rank_exact_all_nonzero}
		Let $\FF \in \{ \CC, \RR\}$.  If $A$ has $\FF$-constant rank, then the syzygy matrix $S$ has the property
		\begin{align*}
			\ker_\CC A(\xi)= \im_\CC S(\xi)  \quad \text{for all } \xi \in \FF^n \setminus \{0\}.
		\end{align*}
	\end{corollary}

	We conclude the section by an example illustrating Theorem \ref{thm:pointwise_exact}. 
	
	\begin{example}
		Let $A$ be the matrix from Example \ref{exa:Euler}. A computation in Macaulay2 gives the syzygy matrix
		$$ S=
		\begin{pmatrix}
			x_{2}x_{3}&0&x_{3}^{2}\\
			-x_{2}^{2}&0&-x_{2}x_{3}\\
			-x_{1}x_{3}&-x_{2}^{2}+x_{3}^{2}&-x_{1}x_{2}\\
			x_{1}x_{2}&-2x_{2}x_{3}&-x_{1}x_{3}\\
			-x_{1}x_{3}&x_{2}^{2}+x_{3}^{2}&x_{1}x_{2}\\
		\end{pmatrix}
		$$
		The generic rank of $A$ is 3 and is attained for example at the point $\xi=(0,1,0)$. A simple computation shows that $\ker_\CC A(\xi)$ is spanned by the two vectors $(0 \ 1 \ 0 \ 0 \ 0 )^\top$ and $(0 \ 0 \ -1 \ 0 \ 1)^\top$, which also span $\im_\CC S(\xi)$. If we instead take the point $\eta=(1,0,0)$ then rank$A(\eta)$ drops to 2. In this case $\ker_\CC A(\eta)$ is a three dimensional space 
		while $\im_\CC S(\eta) = \{0\}$. 
	\end{example}
	
	\section{Applications to the Calculus of Variations}\label{sec:cov}
	We will investigate the lower semicontinuity of variational integrals
	$$
	E[v]\coloneqq\int_{\Omega}F(v(x))\dif x\quad\text{for }A v=0.
	$$
	Under reasonable growth and coercivity conditions on the integrand $F$, the lower semicontinuity of $E$ is sufficient to conclude that we have existence of minimizers for of functionals as above. Problems that involve a coupling a linear differential system with a nonlinear effect are ubiquitous in continuum mechanics \cite{Tartar2005}.
	
	To state our problem clearly, recall the Lebesgue spaces $L^1(\Omega)$ and $L^\infty(\Omega)$, consisting of measurable functions $f\colon\Omega\rightarrow\R$ (or $\CC$) such that
	$$
	\|f\|_{L^1(\Omega)}\coloneqq\int_{\Omega}|f(x)|\dif x<\infty,\quad \|f\|_{L^\infty(\Omega)}\coloneqq
	\underset{x\in\Omega}{\mathrm{ess\,sup}}\, |f(x)|<\infty
	$$
	Here $\Omega\subset\R^n$ is a bounded open set, assumed \emph{convex} for simplicity.  We also recall that a sequence $v_j\wstar v$ in $L^\infty(\Omega)^k$ if and only if
	$$
	\lim_{j\rightarrow\infty}\int_\Omega v_j\cdot g\dif x=\int_\Omega v\cdot g\dif x\quad\text{for all }g\in L^1(\Omega)^k.
	$$
	Throughout this section we will consider $R=\R[x_1,\dotsc,x_n]$ and assume that $A$ is \emph{homogeneous in the rows}, i.e. Assumption~\ref{ass:hom} holds. Let $S\in R^{k\times k'}$ be its syzygy matrix, which is \emph{homogeneous in the columns}: there exist integers $h_j$ such that $S_{ij}$ is $h_j$-homogeneous for each $i=1,\ldots, k$, $j=1,\ldots, k'$.
	
	We can now state the main question that we address in this section:
	\begin{question}
		For which integrands $F$ do we have that
		\begin{align}\label{eq:lsc}\tag{LSC}
			\begin{rcases}
				v_j\wstar v&\text{in }L^\infty(\Omega)^k\\
				A v_j=0
			\end{rcases}\implies
			\liminf_{j\rightarrow\infty}\int_{\Omega}F(v_j(x))\dif x\geq \int_{\Omega}F(v(x))\dif x?
		\end{align}
		In other words, for which nonlinearities $F$, do we have that $E$ is weakly-* sequentially lower semicontinuous on $\ker_{L^\infty(\Omega)^k}A$?
	\end{question}
	The existing results, which we will recall soon, revolve around the notion of $A$-quasiconvexity, as introduced in \cite{FonsecaMuller1999} (see also \cite{Dacorogna1982}). Here we consider a refinement of this notion:
	\begin{definition}
		Let $Q=(0,1)^n$. A continuous integrand $F\colon\R^k\rightarrow \R$ is said to be
		\begin{enumerate}
			\item \textbf{periodically $A$-quasiconvex} if for all $z\in\R^k$ and $v\in C^\infty_{\per}(Q)^k$ with $A v=0$ and $\int_Qv=0$, we have
			$$
			F(z)\leq \int_Q F(z+v(x))\dif x.
			$$
			\item \textbf{compactly supported $A$-quasiconvex} if for all $z\in\R^k$ and $v\in C^\infty_{c}(Q)^k$ with $A v=0$, we have
			$$
			F(z)\leq \int_Q F(z+v(x))\dif x.
			$$
		\end{enumerate}
	\end{definition}
	Periodic $A$-quasiconvexity coincides with the original notion of $A$-quasiconvexity. While it is clear that periodic $A$-quasiconvexity implies compactly supported $A$-quasiconvexity, the converse need not be true in general, as the next example we present shows. On the other hand, the two notions coincide if $A$ has $\R$-constant rank \cite{Raita2019}.
	\begin{example}\label{exa:separate_convexity}
		Let $A_0(v_1,v_2)\coloneqq(\partial_2v_1,\,\partial_1v_2)$. Then the periodic solutions of $A_0 v=0$ are of the form $v_1(x_1,x_2)=f_1(x_1)$ and $v_2(x_1,x_2)=f_2(x_2)$ where $f_i\in C^\infty_{\per}(0,1)$ have $\int_0^1f_i=0$. On the other hand, there are no nontrivial compactly supported solutions of $A_0v=0$, so that any $F\in C(\R^2)$ is compactly supported $A_0$-quasiconvex. However, $F_0(z_1,z_2)\coloneqq -z_1^2$ is not periodically $A_0$-quasiconvex. To see this, test the inequality at $z=0$ with $v=(f_1(x_1),0)$ where $f_1$ is not identically zero.
		Nevertheless, there exist nontrivial periodically $A_0$-quasiconvex integrands; one can check by hand that $F(z_1,z_2)\coloneqq z_1z_2$ is periodically $A_0$-quasiconvex (in fact, $\pm F$ is). It is not too difficult to show that the class of periodically $A_0$-quasiconvex integrands is given by all functions $F(z_1,z_2)$ that are convex in the $z_1$ and $z_2$ directions. 
	\end{example}
	
	Is not known whether there exist operators $A$ that do  \emph{not} have $\R$-constant rank, but for which compactly supported $A$-quasiconvex integrands are necessarily periodically $A$-quasiconvex. We suspect that the class of operators for which the two classes of integrands coincide is the class of operators $A$ which have $\R$-elliptic uncontrolable part $A_u$.
	
	We can now state the main known result concerning lower semicontinuity of the integrals we are interested in:
	\begin{theorem}[{\cite{FonsecaMuller1999,Raita2019}}]
		Let $A$ have $\R$-constant rank and $F\in C(\R^k)$. Then the following are equivalent:
		\begin{enumerate}
			\item\label{itm:lsc} \eqref{eq:lsc} holds;
			\item\label{itm:pAqc} $F$ is periodically $A$-quasiconvex;
			\item\label{itm:csAqc} $F$ is compactly supported $A$-quasiconvex.
		\end{enumerate}
	\end{theorem}
	In the absence of the $\R$-constant rank condition, the implication \ref{itm:pAqc}$\implies$\ref{itm:lsc} is only known essentially in the example \cite{Muller1999} (see also the generalization in \cite{LeeMullerMuller2009}). In general, we expect to have the following:
	\begin{question}
		Let  $F\in C(\R^k)$. Are the following equivalent?
		\begin{enumerate}
			\item \eqref{eq:lsc} holds;
			\item $F$ is periodically $A$-quasiconvex.
		\end{enumerate}
	\end{question}
	Answering this question is well beyond the scope of this paper. Instead, we can investigate circumstances when compactly supported $A$-quasiconvexity is sufficient for lower semicontinuity. We will make the following technical assumption on the syzygy matrix $S$ of $A$: 
	\begin{assumption}\label{ass}
		For any sequence $u_j $ such that $Su_j\wstar Su$ in $ L^\infty(\Omega)^k$ as $j\rightarrow\infty$, we have that there exists $\tilde u_j$, $\tilde  u$ such that $Su_j=S\tilde u_j$, $Su=S\tilde u$ and $\partial^\alpha \tilde u_j^i\rightarrow \partial^\alpha \tilde u^i$ in $L^\infty(\Omega)$ as $j\rightarrow\infty$ for each  $i=1,\ldots, k'$ and $|\alpha|<h_i$.
	\end{assumption}
	Here $y^i$ denotes the $i$th entry of $y\in\R^{k'}$. The following is our main analytic result:
	\begin{theorem}\label{thm:lsc_main}
		Let  $F\in C(\R^k)$ and $A$ be a differential operator with syzygy matrix $S$ such that the compactness Assumption \ref{ass} holds. Then the following are equivalent:
		\begin{enumerate}
			\item\label{itm:main_lsc} \eqref{eq:lsc} holds if $v_j\in C_c^\infty(\Omega)^k$;
			\item\label{itm:main_qc} $F$ is compactly supported $A$-quasiconvex.
		\end{enumerate}
		Alternatively, if $A$ is controllable in $C^\infty$, we can replace \ref{itm:main_lsc} with
		\begin{enumerate}  \setcounter{enumi}{2}
			\item\label{itm:lsc_controllable} \eqref{eq:lsc} holds if $v_j\in C^\infty(\Omega)^k$.
		\end{enumerate}
	\end{theorem}
	The necessity part is classical and is covered, e.g., by the proof in \cite[Thm.~3.6]{FonsecaMuller1999}. The fact that compactly supported $A$-quasiconvexity implies either \ref{itm:main_lsc} or \ref{itm:lsc_controllable} follows immediately from the following:
	\begin{proposition}\label{prop:lsc}
		Let $S\in R^{k\times k'}$ be a differential operator such  that each column $(S_{\cdot i})$ is $h_i$-homogeneous and which satisfies Assumption \ref{ass}. Let $F\in C(\R^k)$ be such that
		\begin{align}\label{eq:qc_S}
			\int_QF(z+Su(x))-F(z)\dif x\geq 0\quad\text{for }z\in\R^k,\,u\in C_c^\infty(Q)^{k'},
		\end{align}
		where $Q=(0,1)^n$. We then have the lower semicontinuity implication
		$$
		Su_j\wstar Su\text{ in }L^\infty(\Omega)^k\implies\liminf_{j\rightarrow\infty}\int_{\Omega}F(Su_j(x))\dif x\geq \int_{\Omega}F(Su(x))\dif x,
		$$
		where $u_j\in C^\infty(\Omega)^{k'}$.
	\end{proposition}
	\begin{proof}[Proof of Theorem~\ref{thm:lsc_main}]
		We already argued that we need only prove that the quasiconvexity assumption implies the lower semicontinuity statements. By \Cref{thm:exact_Ccinfty}, we have that compactly supported $A$-quasiconvexity implies \ref{itm:main_qc}. By \Cref{thm:exact_Ccinfty}, resp. by controllability, we have that the competitor maps in both \ref{itm:main_lsc}, resp. \ref{itm:lsc_controllable} are of the form $v_j=Su_j$. The result then follows from Proposition~\ref{prop:lsc}.
	\end{proof}
	It remains to prove Proposition~\ref{prop:lsc}. This will require a few technical preliminaries. We will write
	$$
	S(\xi)y=\sum_{l=1}^{k'} S^l(\xi)y^l\quad\text{for }\xi\in\R^n,\,y\in\R^{k'},
	$$
	where each $S^l\in R^{k}$ is a $h_l$-homogeneous differential operator acting on scalar fields. We define its \emph{image wave cone} by
	$$
	\mathcal  W^S\coloneqq \bigcup_{\xi\in\R^n}\im S(\xi).
	$$
	\begin{lemma}\label{lem:range_S}
		We have that $\{Su(0)\colon u\in C^\infty(\R^n)^{k'}\}=\mathrm{span\,}\mathcal W^S$.
	\end{lemma}
	We will thus regard the set above as the \emph{essential range} of $S$.
	\begin{proof}
		We can equivalently prove that $\{Su(0)\colon u\in C_c^\infty(\R^n)^{k'}\}=\mathrm{span\,}\mathcal W^S$. In that case, $\widehat{Su}$ is a Schwartz function, where $\hat{\,\cdot\,}$ denotes the Fourier transform,  so 
		$$
		Su(0)=\int_{\R^n}\widehat{Su}(\xi)\dif \xi=\int_{\R^n} \sum_{l=1}^{k'}(-\imag)^{h_l}S^l(\xi)\hat u^l(\xi)\dif \xi=\int_{\R^n} \sum_{l=1}^{k'}S^l(\xi)\Re[(-\imag)^{h_l}\hat u^l(\xi)]\dif \xi.
		$$
		By setting $y^l(\xi)\coloneqq \Re[(-\imag)^{h_l}\hat u^l(\xi)]$, we have that 
		$$
		Su(0)=\int_{\R^n} S(\xi) y(\xi)\dif\xi,
		$$
		which clearly implies $Su(0)\in\mathrm{span\,}\mathcal W^S$.
		
		Conversely, let $\xi\in\R^n$, $y\in \R^{k'}$ and look at $S(\xi)y$. Define the function
		$$
		u(x)\coloneqq\exp(x\cdot\xi) y,
		$$
		so that $Su(x)=\sum_{l=1}^{k'}\exp(x\cdot\xi)S^l(\xi)y^l=\exp(x\cdot\xi)S(\xi)y$. It follows that $Su(0)=S(\xi)y$, which shows that $\mathcal W^S\subset\{Su(0)\colon u\in C^\infty(\R^n)^{k'}\}$. Since the right hand side is a linear space, the conclusion follows.
	\end{proof}
	\begin{lemma}\label{lem:lipschitz}
		If \eqref{eq:qc_S} holds, then $F$ is convex in the directions of $\mathcal W^S$. In particular, $F|_{\mathrm{span\,}\mathcal W^S}$ is locally Lipschitz.
	\end{lemma}
	\begin{proof}
		The proof is standard. For the directional convexity, one employs a modification of the proof of \cite[Thm.~4.2]{KirchheimKristensen2016}. The prof of the Lipschitz regularity is elementary, see \cite[Lem.~2.3]{KirchheimKristensen2016}.
	\end{proof}
	\begin{lemma}\label{lem:indep_qc}
		If \eqref{eq:qc_S} holds for $Q=(0,1)^n$, then it holds for any cube $P\subset\R^n$.
	\end{lemma}
	\begin{proof}
		Let $w\in C_c^\infty(P)^{k'}$. There exists a translation and dilation such that $x_0+\varepsilon P\subset Q$. In this case, we consider the rescaling
		$$
		u^l(x)=\varepsilon^{h_l} w^l\left(\frac{x-x_0}{\varepsilon}\right)\in C_c^\infty(Q),\quad l=1,\ldots k'.
		$$
		We then have that \eqref{eq:qc_S} implies
		\begin{align*}
			F(z)&\leq \int_Q F(z+Su(x))\dif x=|Q\setminus (x_0+\varepsilon P)|F(z)+ \int_{x_0+\varepsilon P} F(z+S u(x)) \dif x\\
			&=(1-\varepsilon^n|P|)F(z)+\int_P F(z+w(y))\varepsilon^n\dif y,
		\end{align*}
		where we made the change of variable $x=x_0+\varepsilon y$. Rearranging proves the claim.
	\end{proof}
	\begin{remark}
		We will see that for the conclusion of Proposition~\ref{prop:lsc} to hold, it suffices to require that $z\in \mathrm{span\,}\mathcal W^S$ in \eqref{eq:qc_S}.
	\end{remark}
	\begin{proof}[Proof of Proposition~\ref{prop:lsc}]
		By Assumption \ref{ass}, we can assume that $\partial^\alpha  u_j^i\rightarrow \partial^\alpha  u^i$ in $L^\infty(\Omega)$ as $j\rightarrow\infty$ for each  $i=1,\ldots, k'$ and $|\alpha|<h_i$ (recall that $y^i$ denotes the $i$th entry of $y\in\R^{k'}$). This is because replacing $u_j$ with $\tilde u_j$ does not change the implication we want to prove. Let $\varepsilon>0$.
		
		We consider a regular grid of cubes of side length $m^{-1}$, $m\rightarrow\infty$ and consider the family $\mathcal F_m$ of those cubes that fit inside $\Omega$. For large enough $m$, we obtain
		$$
		|\Omega\setminus \cup\mathcal F_m|\leq \varepsilon,
		$$
		where the absolute value sign denotes Lebesgue measure on $\R^n$. Making $m$ larger if necessary, we can also assume that
		$$
		\sum_{P\in\mathcal F_m}\int_P |Su(x)-z_P|\dif x\leq \varepsilon,\quad\text{where }z_P\coloneqq\frac{1}{|P|}\int_P Su(x)\dif x.
		$$
		This is so since $Su\in L^\infty(\Omega)^{k}$ is approximated a.e. by $v_m\coloneqq \sum_{P\in\mathcal F_m} z_P\mathbf{1}_{P} $ as $m\rightarrow\infty$. 
		
		We can estimate:
		\begin{align*}
			\int_\Omega F(Su_j)-F(Su)&=\int_{\Omega\setminus\cup\mathcal F_m} F(Su_j)-F(Su) + \sum_{P\in\mathcal F_m} \int_{P}F(Su_j)-F(z_P+Su_j-Su)+\\
			&+\sum_{P\in\mathcal F_m} \int_{P}F(z_P+Su_j-Su)-F(z_P)+\sum_{P\in\mathcal F_m} \int_{P}F(z_P)-F(Su)\\
			&\eqqcolon \mathbf{I}+\mathbf{II}+\mathbf{III}+\mathbf{IV}.
		\end{align*}
		To estimate \textbf{I}, we note that the functions $Su_j$ are uniformly bounded, so we can write $|Su_j|,\,|Su|\leq c_1$ almost everywhere. We then have
		$$
		|\mathbf{I}|\leq 2\max |F(\bar B(0,c_1))||\Omega\setminus \cup\mathcal F_m|\leq c_2\varepsilon,
		$$
		where $\bar B(z,r)$ denotes the closed ball in $\R^k$, centered at $z$, of radius $r>0$.
		
		The next simplest term to estimate is \textbf{IV}, for which we argue as follows: Since $v_m \rightarrow Su $ a.e. as $m\rightarrow\infty$, we also have that $F(v_m)-F(Su)\rightarrow 0$ as $m\rightarrow\infty$. Moreover, $|F(v_m)-F(Su)|\leq c_2$ a.e., so we can infer from the dominated convergence theorem that $|F(v_m)-F(Su)|\rightarrow 0$ strongly in $L^1(\Omega)$. Therefore, we can choose $m$ large enough so that
		$
		|\mathbf{IV}|\leq \varepsilon
		$.
		
		The estimation of \textbf{II} requires Lemma~\ref{lem:lipschitz}, namely that $F$ is locally Lipschitz on the essential range of $S$. We note that $|Su_j|,\,|z_P+Su_j-Su|\leq 3c_1$, so we can consider the Lipschitz constant $L$ of $F$ on $B(0,3c_1)$. We have that
		$$
		|\mathbf{II}|\leq \sum_{P\in\mathcal F_m} \int_{P}L|z_P-Su|\leq L\int_\Omega |v_m-Su|\rightarrow0 \text{ as }m\rightarrow\infty.
		$$
		We can therefore also make $|\mathbf{II}|\leq \varepsilon$ by making $m$ possibly larger.
		
		The three terms we covered so far were estimated in size. It remains to deal with \textbf{III}, which will only be estimated from below. We write $w_j\coloneqq u_j-u$, so that $Sw_j\wstar 0$ and $\partial^\alpha w_j\rightarrow 0$ in $L^\infty$ for $|\alpha|<h_i$, $i=1,\ldots,k'$. The parameter $m$ will be fixed for the remainder of the proof (we chose it such that $|\mathbf{I}|+|\mathbf{II}|+|\mathbf{IV}|\leq c_3\varepsilon$). We focus on one small cube $P\in\mathcal F_m$ and obtain an estimate independent of $m$. 
		
		Write $P_\delta$ for the cube of side length $(1-\delta)m^{-1}$ which has the same center as $P$. We impose $1-\varepsilon\leq (1-\delta)^n$, so that $|P\setminus P_\delta|\leq \varepsilon|P|$. We can construct a cut-off function $\eta_\delta \in C_c^\infty(\R^n)$ with $0\leq \eta_\delta\leq 1$, such that $\eta_\delta =0$ outside $P$, $\eta_\delta =1$ in $P_\delta$ and $|\partial^\alpha\eta_\delta|\leq c_4\delta^{-|\alpha|}$, where $|\alpha|$ is at most equal to the degree of $S$. We define
		$$
		w_{j,\delta}=\eta_\delta w_j\in C_c^\infty(P)^{k'}.
		$$
		We will estimate $Sw_{j,\delta}$ using the product rule. In particular, we want a bound for large $j$ that is independent of $\delta$, for
		$$
		Sw_{j,\delta}=\eta_\delta Sw_j+\sum_{i=1}^{k'}\sum_{|\beta_i|=h_i}\sum_{\alpha<\beta_i} c_{\alpha,i}\partial^\alpha w_j^i\partial^{\beta_i-\alpha}\eta_\delta.
		$$
		Here we use the uniform convergence of the derivatives of $w_j$, regarding $\delta$ as fixed. We can obtain $|Sw_{j,\delta}|\leq c_5$ independently of $\delta$ for large $j$. Using the local boundedness of $F$, we obtain
		\begin{align*}
			\int_P|F(z_P+Sw_j)-F(z_p+Sw_{j,\delta})|&=\int_{P\setminus P_\delta}|F(z_P+Sw_j)-F(z_p+Sw_{j,\delta})|\\
			&\leq |P\setminus P_\delta|c_6\leq c_6 \varepsilon|P|.
		\end{align*}
		We thus obtain that 
		$$
		\liminf_{j\rightarrow\infty}\mathbf{III}\geq \sum_{P\in\mathcal F_m}-c_6\varepsilon|P|+\int_P F(z_P+w_{j,\delta})-F(z_P) \geq -c_6\varepsilon|\Omega|,
		$$
		where we use the quasiconvexity inequality \eqref{eq:qc_S} in $P$. cf. Lemma~\ref{lem:indep_qc}. Collecting all estimates, we have
		$$
		\liminf_{j\rightarrow\infty} \int_{\Omega}F(Su_j)-F(Su)\geq -(c_3+c_6|\Omega|)\varepsilon.
		$$
		Since this holds for any $\varepsilon>0$, the proof is complete.
	\end{proof}
	
	\newpage

\end{document}